\documentclass[11pt]{amsart}

\usepackage{amsmath,amssymb,amsthm,amsfonts,amscd,flafter,
graphicx,verbatim,mathrsfs}
\usepackage[all]{xy}
\usepackage{epstopdf}
\usepackage[colorlinks=false]{hyperref}
\usepackage[all]{hypcap}

\title{Donaldson invariants of symplectic manifolds}

\author[Steven Sivek]{Steven Sivek}
\address{Department of Mathematics \\ Harvard University}
\email{ssivek@math.harvard.edu}

\newcommand\cald{\mathcal{D}}
\newcommand\cpbar{\overline{\mathbb{CP}^2}}
\newcommand\zz{\mathbb{Z}}
\newcommand\rr{\mathbb{R}}
\newcommand\qq{\mathbb{Q}}
\newcommand\cc{\mathbb{C}}
\newcommand\xc{{X^\circ}}

\newcommand\ming{8}

\makeatletter
\let\@@pmod\pmod
\DeclareRobustCommand{\pmod}{\@ifstar\@pmods\@@pmod}
\def\@pmods#1{\mkern4mu({\operator@font mod}\mkern 6mu#1)}
\makeatother

\newtheorem{theorem}{Theorem}[section]
\newtheorem{lemma}[theorem]{Lemma}

\newtheorem{conjecture}[theorem]{Conjecture}
\newtheorem{corollary}[theorem]{Corollary}
\newtheorem{proposition}[theorem]{Proposition}

\theoremstyle{definition}
\newtheorem{definition}[theorem]{Definition}

\newtheorem{remark}[theorem]{Remark}

\makeatletter
\newtheorem*{rep@thm}{\rep@title}
\newcommand{\newreptheorem}[2]{%
\newenvironment{rep#1}[1][0,0]{%
\def\rep@title{#2##1}%
\begin{rep@thm}}%
{\end{rep@thm}}}
\makeatother
\newreptheorem{theorem}{}

\begin{document}
\begin{abstract}
We prove that symplectic 4-manifolds with $b_1=0$ and $b^+ > 1$ have nonvanishing Donaldson invariants, and that the canonical class is always a basic class.  We also characterize in many situations the basic classes of a Lefschetz fibration over the sphere which evaluate maximally on a generic fiber.
\end{abstract}
\maketitle

%
%

\section{Introduction}

The purpose of this paper is to prove a nonvanishing result for Donaldson invariants of symplectic $4$-manifolds.  Donaldson \cite{Donaldson-polynomial} proved that his polynomial invariants are nonzero for large powers of a hyperplane class on a simply connected complex projective surface, and so it is natural to ask if this result generalizes to symplectic manifolds.  For $X$ with symplectic form $\omega$ representing an integral homology class, Donaldson \cite{Donaldson-pencils} showed that some large multiple of $[\omega]$ is Poincar{\'e} dual to the fibers of a Lefschetz pencil on $X$, and so $[\omega]$ is analogous to a hyperplane class on a projective surface.  We therefore prove the following theorem:

\begin{theorem}
\label{thm:nonvanishing}
Let $X$ be a closed $4$-manifold of Donaldson simple type with $b_1(X)=0$ and $b^+(X) > 1$, and let $\omega$ be an integral symplectic form on $X$ with Poincar{\'e} dual $h\in H_2(X;\zz)$.  Then the Donaldson series $\cald^w_X(h)$ is nonzero for any $w\in H^2(X;\zz)$.  In fact, the canonical class $K_X$ is a basic class of $X$, and all basic classes $K$ satisfy
\[ |K \cdot [\omega]| \leq K_X \cdot [\omega] \]
with equality if and only if $K=\pm K_X$.
\end{theorem}

The above inequality was observed by Donaldson \cite{Donaldson-submanifolds}, but here we prove that it is sharp.  The analogous inequality and conditions for equality were proved for Seiberg--Witten basic classes by Taubes \cite{Taubes-more}.  Note that by Witten's conjecture \cite{Witten} and work of Taubes \cite{Taubes-symplectic, Taubes-SWtoGR} we expect that symplectic manifolds with $b_1=0$ and $b^+ > 1$ automatically have Donaldson simple type and their Donaldson and Seiberg--Witten basic classes coincide; in Section \ref{sec:lf-nonsimple} we will prove analogous results in case $X$ does not have simple type.

Theorem \ref{thm:nonvanishing} is a straightforward consequence of the following theorem concerning Donaldson invariants of Lefschetz fibrations over the $2$-sphere.  It was announced in \cite[Section 7.9]{KM-excision} by Kronheimer and Mrowka, whose proof was explained in slides from Kronheimer's talk at the 2009 Georgia International Topology Conference \cite{KM-slides}.  Here we give a new proof using an entirely different strategy, following the ideas of \cite{KM-witten} instead.

\begin{theorem}
\label{thm:lf-nonzero}
Let $X$ be a closed $4$-manifold of Donaldson simple type with $b_1(X)=0$ and $b^+(X) > 1$, and suppose that $X$ admits a relatively minimal Lefschetz fibration over $S^2$ with generic fiber $\Sigma$ of genus $g \geq 2$.  Let $w \in H^2(X;\zz)$ be any class whose pairing with $h=[\Sigma]$ is odd.  Then the Donaldson series $\cald^w_X(h)$ is nonzero and has leading term of order $e^{2g-2}$.
\end{theorem}

We will first establish Theorem \ref{thm:lf-nonzero} for all $g\geq \ming$, deduce Theorem \ref{thm:nonvanishing} as a corollary, and then use Theorem \ref{thm:nonvanishing} to prove the remaining cases of Theorem \ref{thm:lf-nonzero}.  If $X$ does not have simple type, we will prove analogously (Theorem \ref{thm:lf-nonzero-nonsimple}) that there is some $c\neq 0$ for which $D^w_X(h^n)$ is asymptotic to $c(2g-2)^n$ for all large $n\equiv -w^2-\frac{3}{2}(b^+(X)+1)\pmod{4}$.

In Section \ref{sec:lf-basic-classes} we will prove the following theorem, which may be of independent interest; the analogous result for the Heegaard Floer $4$-manifold invariants was proved by Ozsv{\'a}th and Szab{\'o} \cite{OS-symplectic}, but we could not find a complete proof of the Seiberg--Witten version in the literature.

\begin{theorem}
Let $X\to S^2$ be a relatively minimal Lefschetz fibration with generic fiber $\Sigma$ of genus $g \geq 2$, and suppose that $b^+(X) > 1$.  Then the canonical class $K_X$ is the unique Seiberg--Witten basic class of $X$ satisfying $K\cdot \Sigma = 2g - 2$.
\end{theorem}

We cannot prove the analogous uniqueness result for Donaldson basic classes, assuming that $b_1(X)=0$ as well.  However, we will show that if a basic class $K$ satisfies $K\cdot \Sigma = 2g-2$ then either $K=K_X$ modulo torsion or $K^2 < K_X^2$, which we expect to be impossible; and we will also show (Proposition \ref{prop:lf-nonminimal-uniqueness}) that if every component of every fiber intersects some section of square $-1$, then $K\cdot \Sigma = 2g-2$ implies that $K=K_X$ up to torsion.  This last statement is in fact strong enough to complete the proof of Theorem \ref{thm:nonvanishing}.

As mentioned above, our approach to Theorems \ref{thm:nonvanishing} and \ref{thm:lf-nonzero} is based on the strategy used by Kronheimer and Mrowka \cite{KM-witten} in their celebrated proof of the Property P conjecture.  They used some specific cases of Witten's conjecture \cite{Witten} relating the Donaldson and Seiberg--Witten invariants, together with known facts about Seiberg--Witten theory, to prove that certain symplectic manifolds had nonzero Donaldson invariants; this immediately implied nonvanishing results for the instanton Floer homology of some $3$-manifolds which separate them.  In this paper we use the same known cases of Witten's conjecture (applying it to a slightly larger class of $4$-manifolds) together with a gluing theorem of Mu{\~n}oz \cite{Munoz-gluing} which determines the Donaldson invariants of a fiber sum in terms of the invariants of each summand.  In particular, all of our techniques were available at the time \cite{KM-witten} was published.  Kronheimer and Mrowka's proof \cite{KM-excision, KM-slides} of Theorem \ref{thm:lf-nonzero} proceeds instead via excision for instanton Floer homology, following the strategy used by Ozsv{\'a}th and Szab{\'o} in \cite{OS-symplectic}.

The organization of this paper is as follows.  In Section \ref{sec:background} we review the necessary background on Donaldson invariants, simple type, and the relation to Seiberg--Witten invariants.  Section \ref{sec:lf-basic-classes} is devoted to studying the basic classes of Lefschetz fibrations.  In Sections \ref{sec:fiber-sum} and \ref{sec:donaldson-fiber} we prove Theorem \ref{thm:lf-nonzero} for $g\geq \ming$ by applying Mu{\~n}oz's gluing theorem \cite{Munoz-gluing} and known cases of Witten's conjecture to fiber sums of a given Lefschetz fibration with other suitable fibrations.  In Section \ref{sec:proof-nonvanishing} we use the existence of Lefschetz pencils on symplectic manifolds \cite{Donaldson-pencils} to prove Theorem \ref{thm:nonvanishing}, and in Section \ref{sec:proof-lf-nonzero} we complete the proof of Theorem \ref{thm:lf-nonzero}.  Finally, in Section \ref{sec:lf-nonsimple} we discuss the analogous results for symplectic manifolds which do not have simple type.

\subsection*{Acknowledgments}

This work was originally motivated by the desire to prove Proposition \ref{prop:lf-relative-invariant} for a joint project with John Baldwin, who I would like to thank for his encouragement.  I am grateful to him as well as Aliakbar Daemi, Peter Kronheimer, Tom Mrowka, and Jeremy Van Horn-Morris for helpful conversations.  This work was supported by NSF postdoctoral fellowship DMS-1204387.

%
%

\section{Background on Donaldson invariants}
\label{sec:background}

%
%

\subsection{Basic classes and simple type}

Suppose that $X$ is a closed $4$-manifold with $b_{1}(X)=0$ and $b^{+}(X)>1$ odd and a fixed homology orientation, and fix a class $w\in H^{2}(X;\zz)$.  Let $\mathbb{A}(X)$ be the graded symmetric algebra on $H_{2}(X;\rr)\oplus H_{0}(X;\rr)$, where the positive generator $x\in H_{0}(X;\zz)$ has degree 4 and $H_{2}(X;\rr)$ lies in degree $2$. Then the Donaldson invariants \cite{Donaldson-polynomial, DK-book} corresponding to $U(2)$ bundles $E\to X$ with $c_{1}(E)=w$ form a linear map $D_{X}^{w}:\mathbb{A}(X)\to\rr$ as in \cite{KM-embedded}, and this map is zero on any homogeneous element of degree $d$ unless $d\equiv -2w^2 - 3(b^+(X)+1)\pmod{8}$.
\begin{definition}[\cite{KM-embedded}]
The manifold $X$ has \emph{simple type} if $D_{X}^{w}(x^{2}z)=4D_{X}^{w}(z)$ for all $z\in\mathbb{A}(X)$.
\end{definition}
Many complex surfaces with $b_1=0$ and $b^+ > 1$ are known to have simple type, including elliptic surfaces and complete intersections.  In fact, we expect all symplectic manifolds with $b_1=0$ and $b^+ > 1$ to have simple type, by Witten's conjecture (\cite{Witten}, see also Section \ref{ssec:witten-conj}) and work of Taubes \cite{Taubes-SWtoGR} in Seiberg--Witten theory.

For manifolds of simple type, given $h\in H_2(X;\zz)$ we can form the power series 
\[
\cald_X^w(h) = D_X^w\left(\left(1+\frac{x}{2}\right)e^h\right) = \sum_{i=0}^\infty \frac{D_X^w(h^i)}{i!} + \frac{1}{2}\sum_{i=0}^\infty \frac{D_X^w(xh^i)}{i!}.
\]
Kronheimer and Mrowka proved the following structure theorem:
\begin{theorem}[{\cite[Theorem 1.7]{KM-embedded}}]
\label{thm:structure-simple-type}
If $X$ has simple type and Betti numbers $b_{1}(X)=0$ and odd $b^{+}(X)>1$, then there are finitely many classes $K_{1},\dots,K_{s}\in H^{2}(X;\zz)$ such that
\[
\cald_{X}^{w}(h)=\exp\left(\frac{Q(h)}{2}\right)\sum_{r=1}^{s}(-1)^{(w^{2}+K_{r}\cdot w)/2}\beta_{r}e^{K_{r}\cdot h}
\]
as a function on $H^{2}(X;\rr)$. Here $Q$ is the intersection form of $X$, viewed as a quadratic function, and the $\beta_{r}$ are nonzero rational numbers.
\end{theorem}

The $K_{r}$ are called the \emph{basic classes} of $X$, or possibly the Donaldson basic classes to emphasize that these are not necessarily the same as the Seiberg--Witten basic classes.  We remark that the Donaldson basic classes are only well-defined up to torsion, so throughout this paper we will actually think of them as elements of $H^2(X;\zz)/\mathrm{torsion}$.

\begin{proposition}[{\cite[Theorem 8.1]{KM-embedded}}]
\label{pro:tight-implies-simple}
If $b_{1}(X)=0$ and $b^{+}(X)>1$ is odd, and $X$ contains an embedded surface $S$ with $[S]^{2}=2g(S)-2 > 0$, then $X$ has simple type.
\end{proposition}
Such a surface is sometimes called a \emph{tight} surface. For example, if $X$ as above is symplectic and contains a closed Lagrangian surface $L$ of genus at least $2$, then $L$ is necessarily tight and so
$X$ has simple type.  Similarly, if $X$ has a Lagrangian torus in an odd homology class, then it must have simple type: this follows from a result of Mu{\~n}oz \cite[Proposition 9.3]{Munoz-fukaya-floer}, which says in part that if $b_1(X)=0$ and $b^+(X) > 1$, and $X$ contains a homologically odd surface $\Sigma$ of self-intersection $0$ and genus at most $2$, then $X$ has simple type.  The cited result also implies that $X$ has simple type if it admits a genus 2 Lefschetz fibration over $S^2$.

Finally, we remark that Mu{\~n}oz \cite{Munoz-fukaya-floer} has shown that all $4$-manifolds with $b_1=0$ and $b^+ > 1$ have {\em finite type} of some order $n\geq 0$, meaning that $D^w_X((x^2-4)^n z) = 0$ for all $z$.  In this case there is still a notion of basic class and related structure theorem \cite{Munoz-nonsimple, Munoz-donaldson-nonsimple}.  These basic classes satisfy an adjunction inequality \cite[Theorem 1.7]{Munoz-higher} of the form $|K\cdot\Sigma| + \Sigma^2 + 2d(K) \leq 2g-2$ whenever $\Sigma$ is an embedded surface of genus $g\geq 1$ and either $\Sigma^2 > 0$ or $\Sigma^2=0$ and $\Sigma$ represents an odd homology class; here $d(K)\geq 0$ is the ``order of finite type'' of $K$.  (When $X$ has simple type, all basic classes have $d(K)=0$ and so this reduces to the usual adjunction inequality $|K\cdot \Sigma| + \Sigma^2 \leq 2g-2$ as in \cite{KM-embedded}.)

%
%

\subsection{Witten's conjecture}
\label{ssec:witten-conj}

Witten \cite{Witten} conjectured the following relationship between Donaldson and Seiberg--Witten invariants.

\begin{conjecture}
\label{con:witten-conjecture}
Suppose that $X$ is a closed, oriented $4$-manifold with $b_1(X)=0$ and $b^+(X) > 1$ odd.  If $X$ has Seiberg--Witten simple type, then it has Donaldson simple type as well, the Seiberg--Witten and Donaldson basic classes coincide, and 
\[ \cald^w_X(h) = c(X) \exp\left(\frac{Q(h)}{2}\right) \sum_{r=1}^s (-1)^{(w^2+K_r \cdot w)/2} SW(K_r) e^{K_r\cdot h}, \]
where $K_1,\dots,K_s$ are the basic classes of $X$ and $c(X)$ is a nonzero rational number.
\end{conjecture}

The full conjecture also predicts that $c(X) = 2^{2+(7e+11\sigma)/4}$ where $e$ and $\sigma$ are the Euler characteristic and signature of $X$.  In any case, we note that this is exactly the formula of Theorem \ref{thm:structure-simple-type} with $\beta_r = c(X)\cdot SW(K_r)$.

Using work of Feehan and Leness \cite{FL}, Kronheimer and Mrowka established the following special case of Witten's conjecture: 
\begin{theorem}[{\cite[Corollary 7]{KM-witten}}]
\label{thm:witten-conjecture}
Suppose that $X$ as above contains a tight surface and a sphere of self-intersection $-1$, and that $X$ has the same Euler characteristic and signature as a smooth hypersurface in $\mathbb{CP}^3$ of even degree at least 6.  Then $X$ satisfies Conjecture \ref{con:witten-conjecture}.
\end{theorem}

\begin{remark}
\label{rem:witten-conjecture-general}
The proof of \cite[Corollary 7]{KM-witten} used the following facts about such hypersurfaces $X_*\subset \mathbb{CP}^3$, in addition to the requirements that $b_1(X_*)=0$, $b^+(X_*) > 1$ is odd, and $H^2(X_*;\zz)$ has no $2$-torsion:
\begin{enumerate}
\item The only Seiberg--Witten basic classes are $\pm K_{X_*}$, and $SW(\pm K_{X_*}) = 1$. 
\item The Donaldson invariants of $X_*$ are not identically zero.
\item $X_*$ is spin. 
\item $X_*$ contains a tight surface and a symplectic surface $S$ with $S\cdot S \geq 0$ and $K_{X_*} \cdot S \neq 0$.
\end{enumerate}
We will make use of this observation in Section \ref{sec:proof-lf-nonzero}, by finding other $4$-manifolds which have the same properties and thus also satisfy Conjecture \ref{con:witten-conjecture}, in order to complete the proof of Theorem \ref{thm:lf-nonzero}.
\end{remark}

Finally, in some cases we can prove that Witten's conjecture still holds after blowing down a $4$-manifold which satisfies the conjecture.  In order to do so, we recall how the Donaldson and Seiberg--Witten invariants behave under blowups.

\begin{theorem}[\cite{FS-blowup,KM-embedded}]
\label{thm:donaldson-blowup}
If $X$ has Donaldson simple type, then $\tilde{X}=X\#\cpbar$ does as well and their Donaldson series satisfy 
\begin{eqnarray*}
\cald_{\tilde{X}}^{w}(h) & = & \cald_{X}^{w}(h)\cdot\exp\left(-\frac{(E\cdot h)^{2}}{2}\right)\cosh(E\cdot h)\\
\cald_{\tilde{X}}^{w+E}(h) & = & -\cald_{X}^{w}(h)\cdot\exp\left(-\frac{(E\cdot h)^{2}}{2}\right)\sinh(E\cdot h)
\end{eqnarray*}
where $E$ is dual to the exceptional divisor. If $X$ has basic classes $\{K_{r}\}$, then it follows that the basic classes of $\tilde{X}$ are $\{K_{r}\pm E\}$.
\end{theorem}

\begin{theorem}[\cite{FS-immersed}]
\label{thm:sw-blowup}
If $X$ has Seiberg--Witten simple type with basic classes $\{K_r\}$, then $\tilde{X}=X\#\cpbar$ has basic classes $\{K_r \pm E\}$ where $E$ is dual to the exceptional divisor, and $SW_X(K_r) = SW_{\tilde{X}}(K_r \pm E)$.
\end{theorem}

The following is then a straightforward consequence of Theorems \ref{thm:donaldson-blowup} and \ref{thm:sw-blowup}, as well as the immediate corollary of Fintushel and Stern's blowup formula \cite{FS-blowup} that if $X\#\cpbar$ has Donaldson simple type then so does $X$.
\begin{proposition}
\label{prop:witten-blowdown}
Let $X$ be a $4$-manifold with $b_1 = 0$ and $b^+ > 1$ odd, and suppose that $X$ has Seiberg--Witten simple type.  If $\tilde{X} = X\#\cpbar$ satisfies Conjecture \ref{con:witten-conjecture}, then so does $X$, and furthermore $c(X) = 2c(\tilde{X})$.
\end{proposition}

%
%

\section{Lefschetz fibrations and basic classes}
\label{sec:lf-basic-classes}

Let $X \to S^2$ be a Lefschetz fibration with generic fiber $\Sigma$ of genus $g \geq 2$.  Then $X$ admits a symplectic form $\omega$ for which each fiber is symplectic \cite{Gompf-Stipsicz}, and so it has a canonical class $K_X \in H^2(X;\zz)$.  Throughout this section we will assume that $b^+(X) > 1$, and when we discuss Donaldson invariants we will also assume that $b_1(X)=0$.

Our goal in this section is to determine some strong restrictions on the set of basic classes satisfying $K\cdot \Sigma = 2g-2$, which is maximal by the adjunction inequality $|K\cdot \Sigma| + \Sigma\cdot\Sigma \leq 2g-2$.  The main result for Seiberg--Witten basic classes, Theorem \ref{thm:sw-fibration}, was claimed by Finashin \cite[Section 3]{Finashin}, but the argument given there is incomplete (see Remark \ref{rem:finashin-proof}) so we include a full proof here following the same ideas.  An analogous result was proved by Ozsv{\'a}th and Szab{\'o} \cite[Theorem 5.1]{OS-symplectic} for their $4$-manifold invariants by different means.

Throughout this section we will abuse notation by identifying surfaces inside $X$ with their homology classes, but this should not cause any confusion.

\begin{lemma}
Suppose that $X$ has reducible singular fibers $F_1\cup G_1,\dots,F_k\cup G_k$ where $F_i$ and $G_i$ are the components of the $i$th reducible singular fiber.  Then $\Sigma$, $F_i$, and $G_i$ all are primitive nonzero elements of $H_2(X;\zz)$, both $F_i$ and $G_i$ have self-intersection $-1$, and $\Sigma,F_1,\dots,F_k$ are linearly independent in $H_2(X;\zz)$.
\end{lemma}

\begin{proof}
The fact that $F_i$ and $G_i$ are primitive follows immediately from $F_i \cdot G_i = 1$, and Stipsicz \cite[Theorem 1.4]{Stipsicz} showed that $\Sigma$ is primitive as well.  Since $\Sigma$ is homologous to $F_i \cup G_i$ and disjoint from both $F_i$ and $G_i$, we have $F_i \cdot F_i = F_i \cdot (\Sigma - G_i) = -1$ and likewise for $G_i$.  Finally, if there is a linear relation of the form
\[ n\Sigma + c_1F_1 + \dots + c_kF_k = 0, \]
then pairing both sides with $F_i$ gives $-c_i=0$, so we are left with $n\Sigma=0$ and this implies $n=0$ as well.
\end{proof}

\begin{proposition}
\label{prop:lf-basic-class-subspace}
Let $K$ be a basic class of $X$, either for Seiberg--Witten invariants or for Donaldson invariants if $b_1(X)=0$.  If $K\cdot\Sigma = 2g-2$, then the Poincar{\'e} dual of $K_X-K$ can be expressed up to torsion as a sum $c_0 \Sigma + c_1 F_1 + \dots + c_k F_k$ with $c_i\in \zz$ and $c_0 \geq 0$.
\end{proposition}

\begin{proof}
Since $\Sigma$ and the $F_i$ are primitive and linearly independent, we can extend them to a basis $F_0 = \Sigma, F_1, \dots, F_k, F_{k+1},\dots,F_n$ of $H_2(X;\qq)$.  Let $\alpha_0,\dots,\alpha_n$ be the dual basis of $(H_2(X;\qq))^* \cong H^2(X;\qq)$.

Recall (e.g.\ from \cite[Section 10.2]{Gompf-Stipsicz}) that we can give the Lefschetz fibration $f:X\to S^2$ a symplectic form $\omega = f^*(\omega_{S^2}) + t\eta$, where $\omega_{S^2}$ is an area form on the base and $f^*(\omega_{S^2})$ is Poincar{\'e} dual to $\Sigma$; $[\eta]$ is any integral class on $X$ which evaluates positively on every component of every fiber, including the generic fiber $\Sigma$; and $t>0$ is sufficiently small.  More specifically, the condition on $[\eta]$ means that it must be positive on $\Sigma$ and on each $F_i$ and $G_i$ for $1\leq i\leq k$.

By the inequality $K\cdot[\omega] \leq K_X\cdot[\omega]$, which holds for all basic classes in both Donaldson theory \cite{Donaldson-submanifolds} and Seiberg--Witten theory \cite{Taubes-more}, and the fact that $K\cdot \Sigma = 2g-2 = K_X \cdot \Sigma$ (the latter equality follows from the adjunction formula), we have the necessary condition
\[ (K_X-K)\cdot [\eta] \geq 0. \]
Now suppose that $PD(K_X - K) = \sum_{i=0}^n c_i F_i$ for some $c_i\in\qq$.  Then given any symplectic form on $\omega$ as above, we can replace $[\eta]$ with $[\eta'] = [\eta] - d(c_{k+1}\alpha_{k+1} + \dots + c_n \alpha_n)$ without changing its evaluation on $\Sigma$ or any of the $F_i$ or $G_i = \Sigma-F_i$ for $i\leq k$.  In particular, $[\eta']$ is still positive on these classes, so it yields a new symplectic form and we have
\begin{eqnarray*}
0 \leq (K_X-K) \cdot [\eta'] & = & (K_X - K)\cdot[\eta] - d\left(\sum_{i=k+1}^n c_i \alpha_i\right) \left(\sum_{j=0}^n c_j F_j\right) \\
& = & (K_X-K)\cdot[\eta] - d\sum_{i=k+1}^n c_i^2.
\end{eqnarray*}
Taking $d$ to be arbitrarily large forces the right hand side to be negative unless $\sum_{i=k+1}^n c_i^2 = 0$, so we must have $c_i=0$ for all $i>k$, as desired.

Finally, we note that $[\eta] = d\alpha_0 + \alpha_1 + \dots + \alpha_k$ gives rise to a symplectic form for all $d\geq 2$, and $(K_X-K)[\eta] = c_0d + c_1 + \dots + c_k$ must remain nonnegative for all large $d$, so $c_0 \geq 0$.  Since $\Sigma$ and the $F_i$ are primitive, linearly independent (over $\qq$) integral classes and $K_X-K$ is integral, it also follows that the $c_i$ must all actually be integral as well.
\end{proof}

\begin{remark}
\label{rem:finashin-proof}
The proof of Proposition \ref{prop:lf-basic-class-subspace} was suggested by Finashin in \cite[A3(2)]{Finashin}.  However, the proof there claimed that $[\eta]$ could be any class for which $[\eta]\cdot\Sigma > 0$, and this is only true if there are no reducible singular fibers.  Otherwise we cannot a priori eliminate the possibility that $PD(K_X-K)$ has nontrivial $F_i$--components for some of $F_1,\dots,F_k$.
\end{remark}

\begin{proposition}
Suppose that $X\to S^2$ is relatively minimal and $K$ is a Seiberg--Witten or Donaldson basic class for which $K\cdot \Sigma = 2g-2$.  Then $K^2 \leq K_X^2$, with equality if and only if $K = K_X$ modulo torsion.
\end{proposition}

\begin{proof}
Let $K$ be such a basic class, and write $PD(K_X-K) = n\Sigma + \sum_{i=1}^k c_i F_i$ up to torsion.  If $c_i < 0$ for any $i$, then we observe that $c_i F_i = c_i \Sigma + (-c_i) G_i$, so by replacing $F_i$ in our basis with $G_i$ we may assume without loss of generality that $c_i \geq 0$ for all $i$, and we will still have $n\geq 0$ as in the previous proposition.
We then compute 
\begin{eqnarray*}
K^2 & = & \left(PD(K_X) - n\Sigma - \sum_{i=1}^k c_i F_i\right)^2 \\
& = & K_X^2 - 2n(K_X \cdot \Sigma) - 2\sum_{i=1}^k c_i(K_X \cdot F_i) - \sum_{i=1}^k c_i^2.
\end{eqnarray*}
Now the adjunction formula says that $K_X \cdot S + S\cdot S = 2g(S)-2$ whenever $S$ is a symplectic surface, and by construction $\Sigma$ and each of the $F_i$ are symplectic, so
\begin{eqnarray*}
K_X \cdot \Sigma = 2g-2, & & K_X \cdot F_i = 2g(F_i)-1
\end{eqnarray*}
since $\Sigma$ and $F_i$ have self-intersection $0$ and $-1$ respectively.  In particular, we assumed that $g\geq 2$ and $X\to S^2$ is relatively minimal, which implies $g(F_i) \geq 1$ for each $i$, so both $K_X\cdot\Sigma$ and $K_X\cdot F_i$ are strictly positive.  But this implies that
\[ K_X^2 - K^2 = 2n(K_X \cdot \Sigma) + 2\sum_{i=1}^k c_i(K_X \cdot F_i) + \sum_{i=1}^k c_i^2 \]
is a sum of nonnegative terms which vanish only if $n$ and the $c_i$ are all zero, as desired.
\end{proof}

We remark that nothing we have proved so far in this section requires $X$ to have Donaldson simple type.  Indeed, the only facts we have used about basic classes are the adjunction inequality, which still holds when $X$ has finite type, and the inequality $K\cdot[\omega] \leq K_X\cdot[\omega]$, which follows from the adjunction inequality and the existence of symplectic surfaces dual to $k[\omega]$ for large $k$.

At this point we specialize to Seiberg--Witten theory, where we can apply several theorems of Taubes \cite{Taubes-symplectic, Taubes-more} to completely determine which basic classes evaluate maximally on the fiber $\Sigma$.

\begin{theorem}
\label{thm:sw-fibration}
Let $X\to S^2$ be a relatively minimal Lefschetz fibration of genus $g\geq 2$ satisfying $b^+(X)>1$.  Then the canonical class $K_X$ is the unique Seiberg--Witten basic class $K$ of $X$ for which $K\cdot \Sigma=2g-2$.
\end{theorem}

\begin{proof}
Let $K$ be a Seiberg--Witten basic class with $K\cdot\Sigma=2g-2$, and note that $K_X$ is one such class \cite{Taubes-symplectic}.  We cannot have $K^2 < K_X^2 = 3\sigma(X) + 2\chi(X)$, because then the associated moduli space of monopoles has negative expected dimension, which makes it generically empty and so $K$ would not be a basic class.  (In fact, since $X$ has Seiberg--Witten simple type \cite{Taubes-SWtoGR}, every basic class satisfies $K^2 = K_X^2$.)  Therefore $K_X - K$ is torsion, which implies that $K_X \cdot [\omega] = K \cdot [\omega]$, hence $K=\pm K_X$ by \cite{Taubes-more}. Clearly we cannot have $K=-K_X$, because the fact that $K_X\cdot\Sigma = 2g-2 > 0$ would contradict the claim that $K_X - K = 2K_X$ is torsion, so it follows that $K=K_X$.
\end{proof}

The same argument does not show that a Donaldson basic class $K$ of this form must be the canonical class, because we do not know (although we expect it to be true) that all Donaldson basic classes satisfy $K^2 \geq K_X^2$.  However, in certain situations we can still reach the same conclusion.

\begin{lemma}
\label{lem:maximal-class-section}
Let $X\to S^2$ be a relatively minimal genus $g\geq 2$ Lefschetz fibration with $b_1=0$, $b^+ > 1$,  and generic fiber $\Sigma$, not necessarily of Donaldson simple type, and let $K$ be a basic class of $X$ such that $K\cdot \Sigma = 2g-2$.  Then $K\cdot E = -1$ for any section $E$ of square $-1$.
\end{lemma}

\begin{proof}
The class $\Sigma+E$ is represented by a genus $g$ surface of square $\Sigma^2 + 1 = 1$ obtained by smoothing the point of intersection of $\Sigma\cup E$, and so $K\cdot (\Sigma+E) + 1 \leq 2g-2$ by the adjunction inequality.  Since $K\cdot\Sigma = 2g-2$ by assumption, we see that $K\cdot E \leq -1$.  

We now use a fact proved by Mu{\~n}oz \cite[Remark 3]{Munoz-donaldson-nonsimple} about blow-ups: if $X_0$ is the manifold obtained by blowing down $X$ along $E$, then $K$ has the form $K = L \pm (2n+1)E$ where $L$ is a basic class of $X_0$ and $n\geq 0$, and furthermore $d(K) = d(L)-n(n+1)$.  Since $K\cdot E \leq -1$, we actually have $K = L+(2n+1)E$.  In fact, since $K\cdot \Sigma + \Sigma\cdot \Sigma = 2g-2$ the adjunction inequality \cite[Theorem 1.7]{Munoz-higher} says that $d(K) = 0$, hence $d(L) = n(n+1)$.  If $\Sigma_0 \subset X_0$ is the surface with proper transform $\Sigma\subset X$, so that $\Sigma_0^2 = 1$ and $\Sigma = \Sigma_0 - E$, then
\[ 2g-2 = K\cdot \Sigma = (L+(2n+1)E)\cdot(\Sigma_0 - E) = L\cdot\Sigma_0 + (2n+1) \]
and so the adjunction inequality applied to $L$ and $\Sigma_0$ gives
\[ 2g-2 \geq L\cdot \Sigma_0 + \Sigma_0\cdot\Sigma_0 + 2d(L) = (2g-2 - (2n+1)) + 1 + 2n(n+1) \]
or equivalently $2n^2 \leq 0$.  We conclude that $n=0$, and so $K = L+E$ and $K\cdot E = -1$.
\end{proof}

\begin{proposition}
\label{prop:lf-nonminimal-uniqueness}
Let $X\to S^2$ be a relatively minimal genus $g\geq 2$ Lefschetz fibration for which every component of every fiber intersects at least one section of square $-1$, and suppose that $b_1(X)=0$ and $b^+(X) > 1$. Then any Donaldson basic class $K$ of $X$ such that $K \cdot \Sigma = 2g-2$ must equal $K_X$ up to torsion.
\end{proposition}

\begin{proof}
Let $F_i\cup G_i$ be a reducible singular fiber, and let $E$ be a $(-1)$--section intersecting $F_i$.  Make the section $E$ symplectic by taking the parameter $t$ in the proof of Proposition \ref{prop:lf-basic-class-subspace} sufficiently small.  Then the class $F_i+E$ is represented by a genus $g(F_i)$ surface of square $F_i^2 + 1 = 0$ obtained by smoothing $F_i\cup E$, and so $K\cdot (F_i+E) \leq 2g(F_i)-2$ by the adjunction inequality.  Since both $F_i$ and $E$ are symplectic, the adjunction formula says that $K_X\cdot F_i = 2g(F_i)-1$ and $K_X \cdot E = -1$, and so $K\cdot(F_i+E) \leq K_X\cdot(F_i+E)$.  Therefore $K\cdot F_i \leq K_X \cdot F_i$, and likewise we can show that $K\cdot G_i \leq K_X \cdot G_i$.  Combining the two inequalities, we get
\[ K\cdot\Sigma = K\cdot F_i + K\cdot G_i \leq K_X\cdot F_i + K_X\cdot G_i = K_X\cdot \Sigma, \]
and by assumption the left and right sides are equal, so in fact $K\cdot F_i = K_X\cdot F_i$.  Pairing both sides of $PD(K_X-K) = n\Sigma + \sum c_j F_j$ with $-F_i$, we see that $c_i = 0$.

Finally, we pair both sides of $PD(K_X - K) = n\Sigma$ with a symplectic $(-1)$--section $E$ to get $n = n\Sigma\cdot E = (K_X - K)\cdot E = 0$, and so we conclude that $K_X = K$ as desired.
\end{proof}

\begin{remark}
We have not actually shown in the proof of Proposition \ref{prop:lf-nonminimal-uniqueness} that $K_X$ is a Donaldson basic class, only that $K\cdot \Sigma = 2g-2$ implies $K=K_X$ under the given hypotheses.  However, we will establish that $K_X$ is a basic class for $X$ of simple type during the proof of Theorem \ref{thm:nonvanishing}, and for general $X$ in Corollary \ref{cor:nonvanishing-nonsimple}.
\end{remark}

%
%

\section{A fiber sum which satisfies Witten's conjecture}
\label{sec:fiber-sum}

Let $X\to S^{2}$ be a symplectic Lefschetz fibration with fiber $\Sigma$ of genus $g\geq\ming$, and let $X^{\circ}=X\backslash N(\Sigma)$ be the complement of a neighborhood of a regular fiber. Our goal is to prove the following, cf.\ \cite[Proposition 15]{KM-witten}:
\begin{theorem}
\label{thm:symplectic-embedding}
We can find a symplectic, relatively minimal Lefschetz fibration $Z \to S^2$ with fiber genus $g$ and $b^+(Z) > 1$ so that the fiber sum $W=X\#_{\Sigma}Z$ satisfies $H_{1}(W;\zz)=0$, the restriction map $H^{2}(W;\zz)\to H^{2}(X^{\circ};\zz)$ is surjective, and: 
\begin{enumerate}
\item There is a smooth hypersurface $W_* \subset \mathbb{CP}^3$ of even degree at least $6$ such that $b^-(W) < b^-(W_*)$ and $b^+(W) = b^+(W_*)$.
\item $Z$ contains a tight genus 2 surface disjoint from a generic fiber, hence so does $W$.
\end{enumerate}
\end{theorem}
Much of the proof follows the same lines as in \cite{KM-witten}.  We start by constructing a relatively minimal Lefschetz fibration $Z\to S^{2}$ with fiber $\Sigma$ whose vanishing cycles include a generating set for $H_{1}(\Sigma)$, so that $H_{1}(Z)=0$. Then the union 
\[
W=X^{\circ}\cup_{S^{1}\times\Sigma}Z^{\circ}=X\#_{\Sigma}Z,
\]
where $\#_{\Sigma}$ denotes the Gompf fiber sum \cite{Gompf} along regular fibers $\Sigma$, is also a symplectic manifold admitting a genus $g$ Lefschetz fibration over $S^{2}$. We note that $\Sigma$ is homologically essential in both $X$ and $Z$, since it is a symplectic surface of genus at least $2$, and so the homology class $[\Sigma]$ is primitive in both $X$ and $Z$ by \cite[Theorem 1.4]{Stipsicz}.

We will prove Theorem \ref{thm:symplectic-embedding} by modifying $Z$, and hence $W$, via a sequence of fiber sums.
\begin{lemma}[{cf.\ \cite[Lemmas 11--12]{KM-witten}}]
\label{lem:homology-X}
If $H_{1}(Z)=0$, then $H_{1}(W)=0$ and the restriction map $H^{2}(W)\to H^{2}(X^{\circ})$ is surjective.
\end{lemma}
\begin{proof}
For the first claim, we note that $H_1(Z)=0$ if and only if the vanishing cycles of $Z\to S^2$ generate $H_1(\Sigma)$, and hence the vanishing cycles of $W\to S^2$ do as well.  For the second claim, we have $H^{3}(Z^\circ,\partial Z^\circ)=0$ by Poincar{\'e} duality, and thus $H^3(W,\xc)=0$ by excision, so the restriction $H^{2}(W)\to H^{2}(\xc)$ is surjective by the long exact sequence of the pair $(W,\xc)$.
\end{proof}
\begin{lemma}
\label{lem:fiber-sum-b_pm}
Let $V\to S^{2}$ be a genus $g$ Lefschetz fibration with generic fiber $\Sigma$. If we replace $Z$ with the fiber sum $Z\#_{\Sigma}V$, then this preserves the property $H_{1}(Z)=0$ and hence the conclusion of Lemma \ref{lem:homology-X}. Furthermore, if we let 
\[
n^{\pm}(V)=b^{\pm}(V)-b_{1}(V)+2g-1
\]
then this fiber sum operation increases $b^{+}(W)$ and $b^{-}(W)$ by $n^{+}(V)$ and $n^{-}(V)$, respectively.\end{lemma}
\begin{proof}
The claim $H_{1}(Z)=0$ is immediate, because the vanishing cycles of $Z$ generate $H_{1}(\Sigma)$ and hence the same is true of $Z\#_{\Sigma}V$. To compute $b^{\pm}(\tilde{W})$, where $\tilde{W}=W\#_{\Sigma}V=X\#_{\Sigma}(Z\#_{\Sigma}V)$ is the result of modifying $X$ by this fiber sum, we recall from \cite{Gompf} the formulas 
\begin{eqnarray*}
\sigma(\tilde{W}) & = & \sigma(W)+\sigma(V)\\
\chi(\tilde{W}) & = & \chi(W)+\chi(V)-2\chi(\Sigma).
\end{eqnarray*}
From the second equation and the fact that $b_{1}(\tilde{W})=b_{1}(W)=0$, we have 
\[
2+b^{+}(\tilde{W})+b^{-}(\tilde{W})=(2+b^{+}(W)+b^{-}(W))+(2-2b_{1}(V)+b^{+}(V)+b^{-}(V))-2(2-2g)
\]
and by adding the first equation this simplifies to 
\[
b^{+}(\tilde{W})=b^{+}(W)+b^{+}(V)-b_{1}(V)+2g-1.
\]
Subtracting the first equation from this gives the analogous formula for $b^{-}$.
\end{proof}
\begin{lemma}
\label{lem:construct-fibrations}
If $g\geq\ming$ then there are relatively minimal genus $g$ Lefschetz fibrations $V_{1}\to S^{2}$ and $V_{2}\to S^{2}$ such that $\frac{n^{-}(V_{1})}{n^{+}(V_{1})}<2$ and $\gcd(n^{+}(V_{1}),n^{+}(V_{2})) = 2$.
\end{lemma}

\begin{proof}
Suppose that $g=2k$ is even. Stipsicz \cite[Section 5]{Stipsicz} has constructed a genus $g$ Lefschetz fibration on $V_{1}=(\Sigma_{k}\times S^{2})\#4\cpbar$ by realizing it as the resolution of the double cover of $\Sigma_{k}\times S^{2}$ branched over the singular curve 
\[
(\{p_{1},p_{2}\}\times S^{2})\cup(\Sigma_{k}\times\{q_{1},q_{2}\})
\]
and composing the cover $V_{1}\to\Sigma_{k}\times S^{2}$ with the projection to $S^{2}$; the generic fiber is the double cover of $\Sigma_{k}$ branched at two points, which is indeed $\Sigma_{g}$. Similarly, if we take the branched double cover of $\Sigma_{k-1}\times S^{2}$ with branch locus $(\{p_{1},\dots,p_{6}\}\times S^{2})\cup(\Sigma_{k-1}\times\{q_{1},q_{2}\})$ and resolve singularities, the resulting $V_{2}\cong(\Sigma_{k-1}\times S^{2})\#12\cpbar$ admits a genus $g$ Lefschetz fibration. The resulting triples of Betti numbers $(b_{1},b^{+},b^{-})$ are $(g,1,5)$ and $(g-2,1,13)$ for $V_{1}$ and $V_{2}$ respectively, so we have $\gcd(n^{+}(V_{1}),n^{+}(V_{2}))=\gcd(g,g+2)=2$ and $\frac{n^{-}(V_{1})}{n^{+}(V_{1})}=\frac{g+4}{g}<2$ for all even $g>4$.

Now suppose instead that $g=2k+1$ is odd. Stipsicz \cite{Stipsicz} constructs a genus $g$ Lefschetz fibration on $V_{1}=(\Sigma_{k}\times S^{2})\#8\cpbar$ by the same method as in the even case, using the double cover $\Sigma_{g}\to\Sigma_{k}$ with 4 branch points, and if we use the double cover $\Sigma_{g}\to\Sigma_{k-1}$ with $8$ branch points we get another such fibration on $V_{2}=(\Sigma_{k-1}\times S^{2})\#16\cpbar$.  Then $V_{1}$ and $V_{2}$ have Betti numbers $(b_{1},b^{+},b^{-})$ equal to $(g-1,1,9)$ and $(g-3,1,17)$ respectively, so $\gcd(n^{+}(V_{1}),n^{+}(V_{2}))=\gcd(g+1,g+3)=2$ and $\frac{n^{-}(V_{1})}{n^{+}(V_{1})}=\frac{g+9}{g+1}<2$ for all odd $g>7$.
\end{proof}

It may be possible to construct appropriate $V_1$ and $V_2$ by other means for most values of $g < \ming$; in fact, the above construction already works for $g=6$.  However, it turns out that we cannot do this when $g=2$:

\begin{proposition}
There is no genus $2$ Lefschetz fibration $V_{1}\to S^{2}$ with $\frac{n^{-}(V_{1})}{n^{+}(V_{1})}\leq2$.\end{proposition}
\begin{proof}
Since $n^{\pm}=\frac{e\pm\sigma}{2}+2g-2$, where $e$ and $\sigma$ are the Euler number and signature of $V_{1}$, the condition $\frac{n^{-}}{n^{+}}\leq2$ is equivalent to $3\sigma+e\geq4-4g$. In the case $g=2$, {\"O}zba{\u{g}}c{\i} \cite[Corollary 10]{Ozbagci} has shown that $V_{1}$ must satisfy $c_{1}^{2}\leq6\chi_{h}-3$, where $c_{1}^{2}=3\sigma+2e$ and $\chi_{h}=\frac{\sigma+e}{4}$; but this is equivalent to $3\sigma+e\leq-6$, which makes the desired inequality $3\sigma+e\geq-4$ impossible.
\end{proof}

\begin{lemma}
\label{lem:tight-fibration}
Let $V \to S^2$ be any nontrivial Lefschetz fibration with fiber $\Sigma$ of genus at least $1$.  Then the fiber sum $V \#_\Sigma V$ contains a tight surface of genus 2 which is disjoint from a fiber of the induced fibration.
\end{lemma}

\begin{proof}
Pick two critical values $x_1,x_2$ of the fibration $f: V \#_\Sigma V \to S^2$ which correspond to the same critical value of $V \to S^2$, so that the fibers over them have the same vanishing cycle $c\subset\Sigma$, and let $\alpha \subset S^2$ be a matching path \cite{Seidel-book} with endpoints at $x_1$ and $x_2$.  Then there is a Lagrangian sphere $L$ lying above $\alpha$, and in particular $L\cdot L=-2$.

Now take a closed curve $\gamma \subset S^2$ which intersects $\alpha$ once and separates the sphere into two disks corresponding to the bases of each copy of $V$.  The monodromy along $\gamma$ is trivial, so $f^{-1}(\gamma) \cong S^1\times \Sigma$.  The fiber $\Sigma$ above $\gamma\cap\alpha$ intersects $L$ in $c$, and if $c' \subset\Sigma$ is a curve which intersects $c$ once then $T = S^1 \times c' \subset f^{-1}(\gamma)$ is a torus of self-intersection zero which intersects $L$ in a point.  As remarked in \cite[Corollary 8.5]{KM-embedded}, we can take $L$ and two parallel copies of $T$ and smooth their intersections to get a tight surface of genus $2$ in $V \#_\Sigma V$ which is disjoint from every fiber over a point outside a small neighborhood of $\alpha\cup\gamma$.
\end{proof}

We are now ready to finish proving the main theorem of this section.
\begin{proof}[Proof of Theorem \ref{thm:symplectic-embedding}]
We proceed exactly as in \cite[Lemma 13]{KM-witten}, starting with $W=X\#_{\Sigma}Z$ for an appropriate $Z$ as above, but without having to increase the genus of the given fibration $X\to S^{2}$. Note that any two symplectic manifolds with $b_1=0$ have values of $b^+$ differing by an even number, since $b^+$ is necessarily odd for each of them.

Take two genus $g$ Lefschetz fibrations $V_{1}$ and $V_{2}$ over $S^{2}$ such that $\gcd(n^{+}(V_{1}),n^{+}(V_{2})) = 2$ and $\frac{n^{-}(V_{1})}{n^{+}(V_{1})}<2$, as in Lemma \ref{lem:construct-fibrations}. Then any sufficiently large even number can be written as a nonnegative linear combination $m=k_{1}\cdot n^{+}(V_{1})+k_{2}\cdot n^{+}(V_{2})$ with $k_{2}<n_{+}(V_{1})$. If we replace $Z$ with its fiber sum with $k_{1}$ copies of $V_{1}$ and $k_{2}$ copies of $V_{2}$, the result will be to increase $b^{+}(W)$ by $m$ while keeping $H_{1}(W)=0$, as in Lemma \ref{lem:fiber-sum-b_pm}, and as $m$ gets large we will have $\frac{b^{-}(W)}{b^{+}(W)}\to\frac{n^{-}(V_{1})}{n^{+}(V_{1})}<2$ since $k_{2}$ is bounded independently of $m$.   If we also insist that $k_1 \geq 2$, then it follows that $b^+(Z) > 1$ since $n^+(V_1)$ is positive, and furthermore $Z$ (and hence $W$) contains a tight genus 2 surface by Lemma \ref{lem:tight-fibration}.  In addition, $Z$ is relatively minimal, since the same is true of $V_1$, $V_2$ and the initial choice of $Z$ by construction.

In particular, since smooth hypersurfaces $W_{*}\subset\mathbb{CP}^{3}$ of degree $d$ satisfy $\frac{b^{-}(W_{*})}{b^{+}(W_{*})}\to2$ as $\deg(W_{*})\to\infty$, we can take $d$ to be even and sufficiently large and then this construction will provide $W$ with $b^{-}(W)<b^{-}(W_{*})$ and $b^{+}(W)=b^{+}(W_{*})$, as desired.
\end{proof}

%
%

\section{Donaldson invariants of fiber sums}
\label{sec:donaldson-fiber}

Let $X\to S^{2}$ be a symplectic, relatively minimal Lefschetz fibration with fibers $\Sigma$ of genus $g\geq\ming$. Take a symplectic Lefschetz fibration $Z\to S^2$ as constructed in Theorem \ref{thm:symplectic-embedding}, and let $W=X\#_{\Sigma}Z$ be their fiber sum; note that $W$ is also relatively minimal.  Then $W$ has Donaldson simple type because it contains a tight surface and Seiberg--Witten simple type because it is symplectic, and some blowup of $W$ satisfies Witten's conjecture by Theorem \ref{thm:witten-conjecture}, hence by Proposition \ref{prop:witten-blowdown} so does $W$.  Therefore the Donaldson series of $W$ satisfies
\[
\cald_{W}^{w}(h)=c(W)\exp\left(\frac{Q(h)}{2}\right)\sum_{r=1}^{s}(-1)^{(w^{2}+K_{r}\cdot w)/2}SW(K_{r})e^{K_{r}\cdot h}
\]
where $c(W)$ is a nonzero rational number.

\begin{proposition}
\label{prop:donaldson-nonzero}
If $h\in H^{2}(W;\zz)$ is the class of the fiber $\Sigma$, then $\cald_{W}^{w}(h)$ is nonzero with leading term of order $e^{2g-2}$ for all $w\in H^2(W;\zz)$.
\end{proposition}
\begin{proof}
Let $K_W$ be the canonical class of $W$.  Taubes \cite{Taubes-symplectic} proved that $K_{W}$ is a Seiberg--Witten basic class and $SW(K_W)=\pm1$.  Furthermore, we know from Theorem \ref{thm:sw-fibration} that $K_W$ is the only Seiberg--Witten basic class $K$ for which $K\cdot \Sigma = 2g-2$, and so the coefficient of $e^{2g-2}$ in the series $\cald^w_W(h)$ is equal to $\pm c(W)$, which is nonzero.  Therefore $\cald_W^w(h) \neq 0$, and there are no terms of higher order by the adjunction inequality.
\end{proof}

We now wish to apply a theorem of Mu{\~n}oz concerning Donaldson invariants of fiber sums of manifolds with simple type; this includes $Z$ and $W$ by construction.  We say that classes $w_W\in H^{2}(W;\zz)$, $w_{X}\in H^{2}(X;\zz)$, and $w_{Z}\in H^{2}(Z;\zz)$ are \emph{compatible} if they are all odd when evaluated on $\Sigma$; $w_W$ agrees with $w_{X}$ and $w_{Z}$ on $X^\circ = X\backslash N(\Sigma)$ and $Z^\circ = Z\backslash N(\Sigma)$; and $w_W^{2}\equiv w_{X}^{2}+w_{Z}^{2}\pmod{4}$. This last congruence always holds mod $2$ and can be achieved mod $4$ by replacing $w_W$ with $w_W+PD(\Sigma)$ if necessary.

Next, we define $\mathcal{H}\subset H_{2}(W)$ to be the subspace of all classes $D$ such that $D|_{\partial X^{\circ}}$ is a multiple of the class $[S^{1}\times\{*\}]\in H_{1}(\partial X^{\circ})=H_{1}(S^{1}\times\Sigma)$. Choose a linear map $\mathcal{H}\mapsto H_{2}(X)\oplus H_{2}(Z)$ so that if $D\mapsto(D_{X},D_{Z})$, then $D^{2}=D_{X}^{2}+D_{Z}^{2}$ and $D|_{X^{\circ}}=D_{X}|_{X^{\circ}}$ and likewise for $D$.  Finally, give $X$ and $Z$ homology orientations, and let $W$ have the induced homology orientation as in \cite[Remark 8]{Munoz-gluing}.  Mu{\~n}oz's theorem now says the following, assuming $X$ also has simple type:

\begin{theorem}[{\cite[Theorem 9]{Munoz-gluing}}]
\label{thm:munoz-gluing}
Let $w_W\in H^{2}(W;\zz)$ and pick compatible classes $w_{X}\in H^{2}(X;\zz)$ and $w_{Z}\in H^{2}(Z;\zz)$. Write the Donaldson series for $X$ and $Z$ as 
\begin{eqnarray*}
\cald_{X}^{w_{X}}(\alpha) & = & \exp\left(\frac{Q_{X}(\alpha)}{2}\right)\sum_{j}a_{j,w_{X}}e^{K_{j}\cdot\alpha}\\
\cald_{Z}^{w_{Z}}(\beta) & = & \exp\left(\frac{Q_{Z}(\beta)}{2}\right)\sum_{k}b_{k,w_{Z}}e^{L_{k}\cdot\beta}.
\end{eqnarray*}
Then 
\begin{eqnarray*}
\cald_{W}^{w_W}(tD) & = & \exp\left(\frac{Q_W(tD)}{2}\right)\left(\sum_{K_{j}\cdot\Sigma=L_{k}\cdot\Sigma=2g-2}-2^{7g-9}a_{j,w_X}b_{k,w_Z}e^{(K_{j}\cdot D_{X}+L_{k}\cdot D_{Z}+2\Sigma\cdot D)t}\right.\\
 &  & \left.+\sum_{K_{j}\cdot\Sigma=L_{k}\cdot\Sigma=-(2g-2)}(-1)^{g}2^{7g-9}a_{j,w_X}b_{k,w_Z}e^{(K_{j}\cdot D_{X}+L_{k}\cdot D_{Z}-2\Sigma\cdot D)t}\right).
\end{eqnarray*}
\end{theorem}

\begin{proof}[Proof of Theorem \ref{thm:lf-nonzero} for $g\geq \ming$]
Let $w_X = w$ be a cohomology class with $w\cdot \Sigma$ odd.  Recall from Theorem \ref{thm:symplectic-embedding} that the restriction map $H^2(W;\zz) \to H^2(X^\circ;\zz)$ is surjective, so we can lift $w_X|_{X^\circ}$ to a class $w_W \in H^2(W;\zz)$.  This in turn provides us with a class in $H^2(Z^\circ;\zz)$ by restriction, and this class can be extended across $N(\Sigma)\subset Z$ to an element $w_Z \in H^2(Z;\zz)$ because its restriction to $\partial N(\Sigma)$ agrees with $w_X|_{\partial N(\Sigma)}$, which can itself be extended across $N(\Sigma)$.  Clearly $w_W,w_X,w_Z$ are compatible, except that the congruence $w_W^2 \equiv w_X^2 + w_Z^2$ may only be satisfied mod $2$ rather than mod $4$.  According to \cite[Remark 10]{Munoz-gluing}, this only changes the formula of Theorem \ref{thm:munoz-gluing} by a sign, namely $\epsilon=(-1)^{(g-1)(w_W^2-w_X^2-w_Z^2)/2}$, and since we only wish to prove a nonvanishing result we can ignore this.

Let $D_X$ and $D_Z$ be represented by a generic fiber $\Sigma \subset X^\circ$ and a tight genus 2 surface $T \subset Z^\circ$, respectively, and $D=[\Sigma]+[T]$.  Then $\Sigma \cdot \Sigma = 0$ and $D\cdot D = T\cdot T = 2$, and $K\cdot T = 0$ for every basic class $K$ of $W$ by the adjunction inequality, so if $h \in H_2(W;\zz)$ is the class of a generic fiber then $\cald_W^{w_W}(D) = e\cdot \cald_W^{w_W}(h)$ by Theorem \ref{thm:structure-simple-type}.  By Theorem \ref{thm:munoz-gluing} we have
\[
\frac{\cald_{W}^{w_W}(D)}{e\cdot 2^{7g-9}}=\sum_{K_{j}\cdot\Sigma=L_{k}\cdot\Sigma=2g-2}-a_{j,w_X}b_{k,w_Z}e^{2g-2}+\sum_{K_{j}\cdot\Sigma=L_{k}\cdot\Sigma=-(2g-2)}(-1)^{g}a_{j,w_X}b_{k,w_Z}e^{-(2g-2)}.
\]
By Proposition \ref{prop:donaldson-nonzero}, however, we know that $\cald_{W}^{w_W}(h)$ is an element of $\qq[e^{\pm 1}]$ whose coefficient of $e^{2g-2}$ is $\pm c(W)e^{2g-2}$, and so by equating coefficients we get (up to a sign)
\[ \frac{\pm c(W)}{2^{7g-9}} = \sum_{K_j \cdot \Sigma = L_k \cdot \Sigma = 2g-2} a_{j,w_X} b_{k,w_Z} =
   \left(\sum_{K_j \cdot \Sigma = 2g-2} a_{j,w_X}\right) \left(\sum_{L_k \cdot \Sigma = 2g-2} b_{k,w_Z}\right) \]
which is nonzero.  The two factors in parentheses are exactly the $e^{2g-2}$--coefficients of $\cald_X^{w_X}(h)$ and $\cald_Z^{w_Z}(h)$, so both of these coefficients are nonzero as well.  Therefore the series $\cald_X^{w_X}(h)$ is nonzero, with leading term of order $e^{2g-2}$, and the subset $\{K\mid K\cdot\Sigma = 2g-2\}$ of the basic classes of $X$ is nonempty, as desired.
\end{proof}

%
%

\section{Proof of Theorem \ref{thm:nonvanishing}}
\label{sec:proof-nonvanishing}

Let $(X,\omega)$ and $w\in H^2(X;\zz)$ satisfy the hypotheses of Theorem \ref{thm:nonvanishing}, namely that $X$ has simple type, $b_1(X)=0$, $b^+(X) > 1$, and the class $[\omega]$ is integral.  Donaldson \cite{Donaldson-pencils} proved that for any sufficiently large integer $k$, there is a Lefschetz pencil on $X$ whose generic fibers $\Sigma$ are symplectic submanifolds Poincar{\'e} dual to $k[\omega]$.  If this pencil has reducible singular fibers, then for any such fiber with components $F_i \cup G_i$ we know that $F_i$ is symplectic and so $\Sigma\cdot F_i = k([\omega]\cdot F_i) > 0$, hence $F_i$ contains some base points of the pencil and likewise for $G_i$.

The adjunction formula applied to the fiber $\Sigma$ tells us that if $g$ is the fiber genus, then $2g-2=k^{2}[\omega]^{2} +k(K_X \cdot[\omega])$, and so by taking $k$ large we can insist that $g\geq\ming$. If $\Sigma^{2}=n$, then we can blow up $X$ at the $n>0$ base points of the pencil, and the manifold 
\[ \tilde{X}=X\#n\cpbar \]
will admit a genus $g$ Lefschetz fibration with $n$ sections $E_i$ of self-intersection $-1$, some of which intersect any component of any fiber, and with generic fiber $\tilde{\Sigma}=\Sigma - (E_1+\dots+E_n)$ the proper transform of $\Sigma$.  The $g\geq \ming$ case of Theorem \ref{thm:lf-nonzero} says that $\cald^w_{\tilde{X}}([\tilde{\Sigma}])$ is nonzero for some $w$, hence $\cald^w_X$ is not identically zero by the blowup formula.

Now we know that $\tilde{X}$ has at least one basic class $\tilde{K}$ for which $\tilde{K}\cdot \tilde{\Sigma} = 2g-2$, hence Proposition \ref{prop:lf-nonminimal-uniqueness} ensures that $K_{\tilde{X}}$ is the only such class.  Furthermore, the class $K_{\tilde{X}}$ can be uniquely written as $K + \sum \sigma_i PD(E_i)$, where $K$ is a basic class on $X$ and $\sigma_i = \pm 1$ for each $i$.  Since $K_{\tilde{X}} = K_X + \sum PD(E_i)$, we conclude that $K_X$ is a basic class of $X$.

Finally, suppose that $K$ is some basic class on $X$ for which $K\cdot[\omega] = K_X \cdot [\omega]$, or equivalently $K\cdot[\Sigma] = K_X \cdot [\Sigma]$.  Then $\tilde{K} = K + \sum PD(E_i)$ is a basic class of $\tilde{X}$, and
\[ \tilde{K} \cdot \tilde{\Sigma} = K\cdot\Sigma + n = K_X \cdot \Sigma + n = 2g-2 \]
by the adjunction formula, so it follows that $\tilde{K} = K_{\tilde{X}}$ and then $K = K_X$.  Likewise, if $K\cdot[\omega] = -K_X \cdot [\omega]$ then the basic class $-K$ must be $K_X$, and so $|K\cdot[\omega]| = K_X\cdot [\omega]$ if and only if $K=\pm K_X$.

We conclude that if $h$ is Poincar{\'e} dual to $[\omega]$ and $K_X \cdot [\omega] > 0$, then the Donaldson series
\[ \cald^w_X(h) = \exp\left(\frac{Q(h)}{2}\right) \sum_{r=1}^s (-1)^{(w^2 + K_r\cdot w)/2} \beta_r e^{K_r\cdot h} \]
has exactly one nonzero highest-order term, namely the one where $K_r = K_X$.  It follows that $\cald^w_X(h)$ is nonzero.

If instead $K_X \cdot [\omega] = 0$, then we know that $\pm K_X$ are the only basic classes.  If $K_X$ is nonzero then according to Taubes \cite{Taubes-SWtoGR} it is Poincar{\'e} dual to a nonempty, embedded symplectic curve $S$ (note that this requires $b^+(X) > 1$), hence $K_X \cdot [\omega] = \int_S \omega > 0$ which is a contradiction.  Therefore $K_X = 0$ and we have $\cald^w_X = c\cdot (-1)^{w^2/2}e^{Q/2}$ for some rational $c$.  Since $\cald^w_X$ is not identically zero, we must have $c\neq 0$ and therefore $\cald^w_X(h) \neq 0$ as desired.

\begin{remark}
The condition $K_X \cdot [\omega] = 0$ with $[\omega]$ integral actually forces $X$ to have simple type: if $\Sigma$ is an embedded symplectic surface which is Poincar{\'e} dual to $k[\omega]$ for $k$ large, then $g(\Sigma) \geq 2$ and the adjunction formula says that $\Sigma\cdot \Sigma = 2g(\Sigma)-2$, so $\Sigma$ is a tight surface.
\end{remark}

%
%

\section{Proof of Theorem \ref{thm:lf-nonzero}}
\label{sec:proof-lf-nonzero}

We begin with some observations about the geography of symplectic $4$-manifolds.

\begin{lemma}
For $n\geq 2$, the elliptic surface $E(n)$ contains a tight surface of genus $2$ which is disjoint from both a generic elliptic fiber and a section of self-intersection $-n$.
\end{lemma}

\begin{proof}
It suffices to prove this for $n=2$, since for larger $n$ we know that $E(n)$ is a fiber sum of $E(2)$ and $E(n-2)$ and the $(-n)$--section is obtained by stitching together a $(-2)$--section and a $(-(n-2))$--section of the respective fibrations.  Since $E(2)$ is also a fiber sum $E(1) \#_{T^2} E(1)$, with $(-2)$--section obtained by gluing together $(-1)$--sections of each $E(1)\cong \cc P^2 \# 9\cpbar$, we can construct the desired tight surface just as in Lemma \ref{lem:tight-fibration}.
\end{proof}

\begin{proposition}
\label{prop:witten-conjecture-manifolds}
For any fixed $r < \frac{11}{3}$, there are simply connected, spin, symplectic manifolds $X$ with $\frac{b^-(X)}{b^+(X)} > r$ and $b^+(X)$ arbitrarily large such that:
\begin{enumerate}
\item The only Seiberg--Witten basic classes of $X$ are $\pm K_X$, and $SW(\pm K_X) = 1$.
\item The Donaldson invariants of $X$ are not identically zero.
\item $X$ contains a tight surface and a symplectic surface $S$ with $S\cdot S=0$ and $K_X\cdot S \neq 0$.
\end{enumerate}
\end{proposition}

\begin{proof}
Using the fact that $b^\pm(X) = \frac{e(X) \pm \sigma(X)}{2} - 1$ whenever $b_1(X)=0$, the condition $\frac{b^-}{b^+} > r$ is equivalent to $(r-1)e + (r+1)\sigma < 2(r-1)$.  In terms of the invariants $c_1^2 = 3\sigma+2e$ and $\chi_h = \frac{e+\sigma}{4}$, then, this is also equivalent to 
\[ c_1^2 + (2r-10)\chi_h < r-1. \]

Let $K \subset S^3$ be a knot with Seifert genus $g$, and suppose that the Alexander polynomial $\Delta_K(t)$ has degree $g$.  Fintushel and Stern \cite{FS-nonsymplectic} construct a simply connected, spin $4$-manifold $Z_K$ with Seiberg--Witten simple type and exactly one basic class $\kappa$ up to sign.

The first step in the construction of $Z_K$ is to perform knot surgery \cite{FS-knots} along a generic fiber $T$ of the elliptic fibration on $E(2n)$ to get a manifold $E(2n)_K$.  Inside $E(2n)_K$, there is a genus $g$ surface $S'$ of self-intersection $-2n$, such that $S' \cap (E(2n)\backslash N(T))$ is contained in a $(-2n)$--section of the fibration on $E(2n)$.  Therefore we have a tight genus 2 surface $F \subset E(2n)_K$ which is disjoint from both $S'$ and another generic fiber $T' \subset E(2n)\backslash N(T) \subset E(2n)_K$, hence from the surface $\Sigma'$ of genus $g+n$ and self-intersection $0$ formed by smoothing out $S' + nT'$.  Then $Z_K$ is the fiber sum $E(2n)_K \#_{\Sigma' = C} Y$ for an appropriate symplectic $Y$ with embedded symplectic surface $C$, and $F$ remains a tight genus 2 surface inside $Z_K$.

Now suppose that $K$ is fibered and has genus $g$.  Then the surgered manifold $E(2n)_K$ is symplectic as well, and so $X = Z_K$ admits a symplectic structure.  It follows from Taubes \cite{Taubes-symplectic} that $\kappa$ must be the canonical class, i.e.\ that $SW(\pm K_X) = \pm 1$ and that there are no other basic classes; if we can ensure that $b^+(X)\equiv 3\pmod{4}$, or equivalently that $\chi_h(X)$ is even, then we will have $SW(\pm K_X) = 1$ for the canonical homology orientation on $X$.  The Donaldson invariants of $X$ are nonzero by Theorem \ref{thm:nonvanishing}, and $X$ contains the tight surface $F$ and the symplectic surface $S$ obtained as a parallel copy of $C\subset Y$; then $S^2 = 0$ and $K_X\cdot S = 2g(S)-2$ is nonzero because $g(S) = g+n \geq 2$.

It only remains to determine the invariants $c_1^2(X)$ and $\chi_h(X)$.  According to \cite{FS-nonsymplectic}, this construction yields $c_1^2 = 8(g+n-1)$ and $\chi_h = 3n+g-1$, and so
\[ c_1^2 + (2r-10)\chi_h = (6r - 22)n + (2r-2)(g-1).\]
For $r < \frac{11}{3}$, the coefficient of $n$ is negative, and so for fixed $g$ and any large enough $n$ we will have $c_1^2 + (2r-10)\chi_h < r-1$.  If we also insist that $n+g$ be odd then $\chi_h=\frac{b^+(X)+1}{2}$ will be even, and as $n$ goes to infinity so does $b^+(X)$, as desired.
\end{proof}

We now complete the proof of Theorem \ref{thm:lf-nonzero}.  Let $X$ be a manifold of Donaldson simple type with $b_1(X) = 0$ and $b^+(X) > 1$ odd, and suppose that we have a relatively minimal Lefschetz fibration $X\to S^2$ of genus $g\geq 2$.  By repeating the arguments of Section \ref{sec:fiber-sum}, we can find a relatively minimal genus $g$ Lefschetz fibration $Z\to S^2$ so that the fiber sum $W = X\#_{\Sigma} Z$ has $H_1(W;\zz)=0$, the map $H^2(W;\zz) \to H^2(X^\circ; \zz)$ is surjective, $W$ and $Z$ both contain a tight genus $2$ surface, and we have a manifold $W_*$ as in Proposition \ref{prop:witten-conjecture-manifolds} such that $b^-(W) < b^-(W_*)$ and $b^+(W) = b^+(W_*)$.

To check that we can find such a $W_*$, note that Lemma \ref{lem:construct-fibrations} still provides Lefschetz fibrations $V_1$ and $V_2$ with $\frac{n^-(V_1)}{n^+(V_1)} \leq 3$ and $\gcd(n^+(V_1),n^+(V_2)) = 2$ even when $2\leq g < 8$, and so there is some $N_0$ such that for any odd $N\geq N_0$ we can achieve $b^+(W)=N$ and $\frac{b^-(W)}{b^+(W)} < \frac{7}{2}$.  Now applying Proposition \ref{prop:witten-conjecture-manifolds} with $r=\frac{7}{2}$ guarantees that we can take $W_*$ to have odd $b^+(W_*) > N_0$ and $\frac{b^-(W_*)}{b^+(W_*)} > \frac{7}{2}$ and so obtain the desired $W_*$.

It now follows from Theorem \ref{thm:witten-conjecture} and Remark \ref{rem:witten-conjecture-general}, together with Proposition \ref{prop:witten-blowdown}, that the fiber sum $W=X\#_{\Sigma} Z$ satisfies Conjecture \ref{con:witten-conjecture}.  We proceed exactly as in Section \ref{sec:donaldson-fiber} to conclude that if $w\cdot\Sigma$ is odd, then $\cald^w_X([\Sigma])$ is nonzero with a leading term of order $e^{2g-2}$.

\begin{remark}
\label{rem:lf-poly-growth}
By Theorem \ref{thm:structure-simple-type}, we have proved that for any Lefschetz fibration $X\to S^2$ with fiber class $h=[\Sigma]$ and $w\in H^2(X;\zz)$ satisfying the conditions of Theorem \ref{thm:lf-nonzero}, there is a constant $c\neq 0$ such that $D^w_X(h^d)$ is asymptotic to $c\cdot(2g-2)^d$ for large $d\equiv -w^2 -\frac{3}{2}(b^+(X)+1) \pmod{4}$.  In Donaldson's original notation, this says that $q_{k,X}(h,\dots,h) \sim c'\cdot(2g-2)^{4k}$ for some nonzero $c'$ and all sufficiently large $k$, cf.\ \cite[Theorem C]{Donaldson-polynomial}.
\end{remark}

%
%

\section{Lefschetz fibrations which do not have simple type}
\label{sec:lf-nonsimple}

In this section we use our previous results to study Lefschetz fibrations of genus at least 2 which do not have simple type, and thus prove a nonvanishing theorem for Donaldson invariants of symplectic manifolds in general.  We recall that according to Witten's conjecture, all symplectic $4$-manifolds should have simple type, which would render this section unnecessary.

To $4$-manifolds $X$ with boundary $Y$ and classes $w\in H^2(X;\zz)$ for which the instanton Floer homology group $I_*(Y)_w$ is well-defined (as in \cite{Donaldson-book}, though we follow Kronheimer and Mrowka's notation from \cite{KM-excision} and also confuse $w|_Y$ with the Hermitian line bundle over $Y$ having first Chern class $w$), one can often assign relative Donaldson invariants which satisfy nice gluing theorems \cite{Donaldson-book, BD-gluing}.  In general these associate to elements of $H_2(X,Y)$ with image $\gamma$ in $H_1(Y)$ an element of the Fukaya--Floer homology group $HFF(Y,\gamma)$ \cite{Fukaya}, but in the cases where $\gamma=0$ the relative invariants form a map $\phi^w_X: \mathbb{A}(X) \to I_*(Y)_w$ with values in the ordinary instanton Floer homology of $Y$.

\begin{lemma}
\label{lem:embed-lf-disk}
Let $X\to D^2$ be a Lefschetz fibration of genus at least 2 over the disk.  Then we can extend $X \to D^2$ to a Lefschetz fibration $W \to S^2$ such that $H_1(W;\zz) = 0$, $b^+(W) > 1$, the map $H^2(W)\to H^2(X)$ is surjective, and $W$ has Donaldson simple type, and if $X$ is relatively minimal then so is $W$.
\end{lemma}

\begin{proof}
Let $\Sigma$ denote the generic fiber of $X$ and let $Y=\partial X$, and construct a Lefschetz fibration $Z_0\to D^2$ with fiber $\Sigma$ and boundary $-Y$ such that the vanishing cycles of $Z_0$ are all nonseparating and generate $H_1(\Sigma;\zz)$.  Then we take $W_0 = X \cup_Y Z_0$ and $W = W_0 \#_{\Sigma} W_0 = X \cup_Y (Z_0 \#_{\Sigma} W_0).$

Since $H_1(W_0) = 0$, Lemmas \ref{lem:fiber-sum-b_pm} and \ref{lem:tight-fibration} imply that $b_1(W) = 0$, $b^+(W) > 1$, and $W$ has simple type.  Finally, the surjectivity of $H^2(W)\to H^2(X)$ is implied by $H^3(W,X)=0$, which by excision and Poincar{\'e} duality is equivalent to $H_1(\overline{W\backslash X}) = 0$ and this is immediate from $\overline{W\backslash X} = Z_0 \#_{\Sigma} W_0$.
\end{proof}

Given $(Y,w)$ for which $I_*(Y)_w$ is defined and a closed surface $R\subset Y$, there is a natural operator $\mu(R)$ of degree $-2$ on $I_*(Y)_w$, and we can decompose $I_*(Y)_w$ into the generalized eigenspaces $V_\lambda$ of $\mu(R)$; Kronheimer and Mrowka \cite{KM-excision} define the group $I_*(Y|R)_w$ to be $V_{2g-2}$.  In this notation, the goal of the following proposition is to show that the relative invariant $\phi^w_X(1)$ of a sufficiently nice Lefschetz fibration $X\to D^2$ projects to a nonzero element of $I_*(\partial X|\Sigma)_w$, where $\Sigma$ is a generic fiber.

\begin{proposition}
\label{prop:lf-relative-invariant}
Let $X\to D^2$ be a relatively minimal Lefschetz fibration with generic fiber $\Sigma$ of genus $g\geq 2$ and boundary $Y$.  If $w\cdot\Sigma$ is odd, then the relative invariant $\phi^w_X(1) \in I_*(Y)_w$ has nonzero $V_{2g-2}$--component.
\end{proposition}

\begin{proof}
We extend $X\to D^2$ to a closed, relatively minimal Lefschetz fibration $W \to S^2$ as in Lemma \ref{lem:embed-lf-disk}, and lift $w\in H^2(X;\zz)$ to a class in $H^2(W;\zz)$ which we also denote by $w$.  Let $Z = \overline{W\backslash X}$, and let $h \in H_2(W;\zz)$ denote the class $[\Sigma]$. Then we have an equation
\[ D_W^w(h^n) = \langle\phi_X^w(h^k), \phi_Z^w(h^{n-k})\rangle \]
for the Poincar{\'e} duality pairing $\langle\cdot,\cdot\rangle: I_*(Y)_w \otimes I_*(-Y)_w \to \cc$ (see \cite[Theorem 6.7]{Donaldson-book}).

Define a polynomial $f_0(t)$ by the formula 
\[ f_0(t) = (t+(2g-2))(t^2+(2g-2)^2)\prod_{k=0}^{g-2} (t^4 - (2k)^4) \]
so that the roots of $f_0(t)$ are precisely all numbers of the form $i^r(2k)$ where $0\leq r\leq 3$ and $0\leq k \leq g-1$, except for $2g-2$.  We also let $f_1(t)$ be the characteristic polynomial of the action of $\mu(\Sigma)$ on $I_*(Y)_w$, divided by $(t-(2g-2))^{\dim(V_{2g-2})}$.  Finally, we let $f(t)=f_0(t)f_1(t)$, and we observe that $f(\mu(\Sigma))$ annihilates all generalized eigenspaces $V_\lambda$ except possibly when $\lambda=2g-2$, since the same is true of its factor $f_1(\mu(\Sigma))$.  Thus we can write $f(\mu(\Sigma)) = \psi\circ \pi_{2g-2}$, where $\pi_{2g-2}: I_*(Y)_w \to V_{2g-2}$ is the projection operator and $\psi$ is some endomorphism of $V_{2g-2}$.
Since eigenspaces of $\mu(\Sigma)$ with different eigenvalues are orthogonal with respect to the pairing $\langle\cdot,\cdot\rangle$, it follows that
\begin{eqnarray*}
D_W^w(f(h)) & = & \langle \phi^w_X(1), f(\mu(\Sigma))\cdot \phi^w_Z(1) \rangle \\
& = & \langle \pi_{2g-2}(\phi^w_X(1)), \psi \circ \pi_{2g-2}(\phi^w_Z(1)) \rangle
\end{eqnarray*}
and so $\pi_{2g-2}(\phi^w_X(1))$ (and the pairing $\langle\cdot,\cdot\rangle$ on $V_{2g-2}$) must be nonzero if $D_W^w(f(h))\neq 0$.

Let $d_0 = -w^2 -\frac{3}{2}(b^+(W)+1)$.  Since $W$ has simple type, we can write
\[ 2D_W^w(e^{th}) = \sum_j a_{j,w} e^{t(K_j \cdot h)} + (-i)^{d_0} \sum_j a_{j,w}e^{t(iK_r\cdot h)} \]
as a special case of \cite[Equation 1.10]{KM-embedded} (with $\lambda=0$ and $Q(h)=0$, and writing $a_{j,w} = (-1)^{(w^2+K_j\cdot w)/2}\beta_j$ for convenience).  Comparing $t^k$--coefficients gives us
\[ 2D_W^w(h^k) = \sum_j a_{j,w} (K_j\cdot h)^k + (-i)^{d_0}\sum_j a_{j,w}(iK_j\cdot h)^k \]
and so it is easy to show that 
\[ 2D_W^w(f(h)) = \sum_j a_{j,w}f(K_j\cdot h) + (-i)^{d_0}\sum_j a_{j,w}f(iK_j\cdot h). \]
Now by construction we have $f(K_j\cdot h) = 0$ unless $K_j\cdot h = 2g-2$, and $f(iK_j\cdot h) = 0$ for all $j$; this is because $K_j\cdot h \equiv h^2\pmod{2}$ is even, and the factor $f_0$ of $f$ vanishes on all numbers of the form $i^r\cdot (2k)$ with $0\leq r\leq 3$ and $0 \leq k \leq g-1$ except for $2g-2$.  Therefore
\[ 2D^w_W(f(h)) = f(2g-2) \sum_{K_j\cdot \Sigma = 2g-2} a_{j,w}, \]
and the right hand side is nonzero by Theorem \ref{thm:lf-nonzero} since the sum is the coefficient of $e^{2g-2}$ in $\cald^w_W(h)$, so $\pi_{2g-2}(\phi^w_X(1)) \neq 0$ as desired.
\end{proof}

\begin{remark}
\label{rem:mu-spectrum}
The function $f(t)$ is often more complicated then necessary: in the case where $Y\cong S^1\times\Sigma$ and $w|_Y$ is Poincar{\'e} dual to the $S^1$ factor, Mu{\~n}oz \cite{Munoz-ring-structure} determined the structure of the closely related variant $I_*(S^1\times \Sigma)_{w,\Sigma}$, and in particular the spectrum of the operator $\mu(\Sigma)$ on it.  Kronheimer and Mrowka \cite[Section 7]{KM-excision} observed that as a consequence of these results, the spectrum of $\mu(\Sigma)$ on $I_*(S^1\times\Sigma)_w$ is in this case exactly the set $\{i^r(2k) \mid 0\leq r\leq 3, 0\leq k \leq g-1\}$, so it would have sufficed to take $f = f_1$, and that the generalized eigenspace $V_{2g-2}$ is $1$--dimensional.  Furthermore, since $\mu(\Sigma)$ acts with degree $-2$ on a $(\zz/8\zz)$--graded vector space, they observed that $V_{i\lambda} \cong V_\lambda$ for each eigenvalue $\lambda$ and so $\dim(V_{i^r(2g-2)}) = 1$ for each $r$ as well.  Our argument above does not make any use of this information.
\end{remark}

Finally, in the following theorem we can lift the restriction that $H_1(X) = 0$.  In this case, the Donaldson invariants can be defined on an algebra $\mathbb{A}(X)$ containing the symmetric algebra on $H_0(X) \oplus H_2(X)$.  Furthermore, given a homology class $h\in H_2(X;\zz)$, the Donaldson invariants $D^w_X(h^n)$ can be nonzero only if $n\equiv -w^2 - \frac{3}{2}(b^+(X) - b_1(X) + 1) \pmod{4}$.

\begin{theorem}
\label{thm:lf-nonzero-nonsimple}
Let $X\to S^2$ be a relatively minimal Lefschetz fibration with generic fiber $\Sigma$ of genus $g\geq 2$ such that $b^+(X) > 1$.  Let $\Delta\cong D^2\times\Sigma$ be a small neighborhood of a regular fiber, with boundary $Y\cong S^1\times\Sigma$, and let $w\in H^2(X;\zz)$ be a class for which $w|_Y = PD(S^1)$.  Then there is a nonzero $c$ such that $D^w_X(h^n) \sim c\cdot(2g-2)^n$ for all large $n \equiv d_0 \pmod{4}$, where $d_0 = -w^2 - \frac{3}{2}(b^+(X)-b_1(X)+1)$ and $h$ is the homology class of $\Sigma$.
\end{theorem}

\begin{proof}
Write $\xc = \overline{X\backslash\Delta}$, so that $X = \xc \cup_Y \Delta$.  We remark first that $w$ as in the statement of the theorem exists: certainly we can lift $w|_Y = PD(S^1)$ to a class $w|_\Delta \in H^2(\Delta;\zz)$, and then arguing as in Lemma \ref{lem:embed-lf-disk}, we see that $PD(S^1 \times \{*\}) \in H^2(Y;\zz)$ lifts to some $w|_\xc \in H^2(\xc;\zz)$ if and only if its image in $H^3(\xc,Y) \cong H_1(\xc)$ is zero, and this is true because $H_1(\xc)$ is generated by $H_1(\{*\} \times \Sigma)$.

For $0\leq r\leq 3$ we define
\[ c_r = \langle \pi_{i^r(2g-2)}(\phi^w_\xc(1)), \pi_{i^r(2g-2)}(\phi^w_\Delta(1)) \rangle. \]
Since both $\phi^w_\xc(1)$ and $\phi^w_\Delta(1)$ have nonzero $V_{2g-2}$--component and we have seen that the pairing $\langle\cdot,\cdot\rangle$ is nonzero on $V_{2g-2}$, which is $1$--dimensional (see Remark \ref{rem:mu-spectrum}), we conclude that $c_0 \neq 0$.

Let $p(t)$ denote the product $p_0\cdot \prod_{k=0}^{g-2} (t^4-(2k)^4)^{\dim(V_{2k})}$, where $p_0\in\qq$ is a constant chosen so that $p(2g-2)=1$.  Then we see that $p(\mu(\Sigma))$ acts on $I_*(Y)_w$ by annihilating the generalized eigenspaces $V_\lambda$ with $|\lambda| < 2g-2$, and as multiplication by $p(2g-2)=1$ on each $1$--dimensional eigenspace $V_{i^r(2g-2)}$.  It follows that
\begin{eqnarray*}
D^w_W(p(h)h^n) & = & \langle p(\mu(\Sigma))\mu(\Sigma)^n\cdot \phi^w_\xc(1), \phi^w_\Delta(1) \rangle \\
& = & \sum_{r=0}^3 \langle \pi_{i^r(2g-2)}(\mu(\Sigma)^n\cdot\phi^w_\xc(1)), \pi_{i^r(2g-2)}(\phi^w_\Delta(1)) \rangle \\
& = & (2g-2)^n \sum_{r=0}^3 i^{rn} c_r.
\end{eqnarray*}
It is clear from this that we have $D_W^w(p(h)h^{n+4}) = (2g-2)^4 D_W^w(p(h)h^n)$ for all integers $n$, and furthermore we can express the relations for $0\leq n \leq 3$ by the matrix equation
\[
\left(\begin{array}{l}
D^w_W(p(h)h^0)/(2g-2)^0 \\
D^w_W(p(h)h^1)/(2g-2)^1 \\
D^w_W(p(h)h^2)/(2g-2)^2 \\
D^w_W(p(h)h^3)/(2g-2)^3
\end{array}\right) =
\left(\begin{array}{llll} 1 & 1 & 1 & 1 \\ 1 & i & i^2 & i^3 \\ 1 & i^2 & i^4 & i^6 \\ 1 & i^3 & i^6 & i^9 \end{array}\right)
\left(\begin{array}{l} c_0 \\ c_1 \\ c_2 \\ c_3 \end{array}\right).
\]
The $4\times 4$ matrix is invertible, and since the $c_i$ are not all zero it follows that for some $n$ we must have $D_W^w(p(h)h^n) \neq 0$.  In fact, we must have $n\equiv d_0\pmod{4}$, since otherwise every monomial $h^e$ of $p(h)h^n$ has degree $e\equiv n \not\equiv d_0\pmod{4}$ and so $D_W^w(h^e)=0$.  Therefore
\begin{eqnarray*}
D^w_X(p(h)h^n) = C(2g-2)^n & \mbox{for all }n\equiv d_0\pmod*{4}
\end{eqnarray*}
where $C$ is a nonzero constant, and $D^w_X(p(h)h^n) = 0$ for all other $n$.

Consider the generating function $F(t) = \sum_{j=0}^\infty D_X^w(h^j)t^j$.  If we write $p(t) = \sum_{j=0}^d a_j t^j$, where $d=\deg(p)$, then for any $j\geq 0$ the $t^{d+j}$--coefficient of
\[ (a_d + a_{d-1}t^1 + \dots + a_0t^d)\cdot F(t) = t^d p\left(\frac{1}{t}\right) \cdot F(t) \]
is equal to $D_X^w(a_dh^{d+j} + a_{d-1}h^{d-1+j} + \dots + a_0h^j)$, or $D_X^w(p(h)h^j)$.  Thus there is a polynomial $q$ of degree less than $d+4$ such that
\[ (1-(2g-2)^4t^4)\cdot t^d p\left(\frac{1}{t}\right)\cdot F(t) = q(t), \]
and since $t^d p(1/t) = p_0\cdot\prod_{i=0}^{g-2} (1-(2k)^4t^4)^{\dim(V_{2k})}$ we can solve for $F$ and expand into partial fractions of the form
\[ F(t) = \frac{a_0 + a_1t + a_2t^2 + a_3t^3}{1-(2g-2)^4t^4} + \left(r(t) + \sum_{k=0}^{g-2} \frac{q_k(t)}{(1-(2k)^4t^4)^{\dim(V_{2k})}}\right) \]
where the $a_i$ are constants and $r, q_k$ are polynomials.  If $a_{d_0} = 0$ (interpreting the subscript modulo 4) then for all large $n\equiv d_0\pmod{4}$ we must have $|D_X^w(h^n)| \ll (2g-2)^n$, and so $D_X^w(p(h)h^n) = C(2g-2)^n$ cannot hold for nonzero $C$.  This is a contradiction, so $a_{d_0} \neq 0$ and $D_X^w(h^n) \sim a_{d_0}(2g-2)^n$ for all large $n\equiv d_0\pmod{4}$.
\end{proof}

\begin{remark}
Kronheimer and Mrowka \cite{KM-excision} also showed that for any $\Sigma$--bundle $Y \to S^1$ with fiber genus $g\geq 2$ and class $w\in H^2(Y;\zz)$ such that $w\cdot\Sigma$ is odd, the $(2g-2)$--eigenspace of $\mu(\Sigma)$ on $I_*(Y)_w$ is $1$--dimensional.  Thus we may repeat the proof of Theorem \ref{thm:lf-nonzero-nonsimple} verbatim to show that its conclusion still holds for any class $w\in H^2(X;\zz)$ with $w\cdot\Sigma$ odd.
\end{remark}

\begin{corollary}
\label{cor:nonvanishing-nonsimple}
Let $(X,\omega)$ be a symplectic manifold with $b_1=0$ and $b^+ > 1$.  If $h\in H_2(X;\zz)$ is Poincar{\'e} dual to the class of an integral symplectic form, then the Donaldson invariants $D_X^w(h^n)$ of $X$ are nonzero for large $n$ congruent to $d_0 = -w^2-\frac{3}{2}(b^+(X)+1) \pmod{4}$, and $K_X$ is a basic class.
\end{corollary}

\begin{proof}
We mostly repeat the proof in the case where $X$ has simple type:  Take a Lefschetz pencil $X\to S^2$ with fibers $\Sigma$ in the class $kh$ for $k$ large and having genus $g\geq 2$.  Let $\tilde{X} \to S^2$ be the Lefschetz fibration obtained by blowing up $X$ at the $(kh)^2$ base points of the pencil and having generic fiber $\tilde{\Sigma}$.  Then for an appropriate choice of $\tilde{w}$ we have $D^{\tilde{w}}_{\tilde{X}}(\tilde{\Sigma}^n) \sim c\cdot (2g-2)^n$ for nonzero $c$ and all $n \equiv d_0\pmod{4}$.  Hence by \cite[Theorem 6]{Munoz-nonsimple} there is at least one basic class $\tilde{K}$ on $\tilde{X}$ for which $\tilde{K}\cdot \tilde{\Sigma} = 2g-2$.  Since at least one $(-1)$--section intersects every component of every fiber of $\tilde{X} \to S^2$, Proposition \ref{prop:lf-nonminimal-uniqueness} says that $\tilde{K} = K_{\tilde{X}}$, and since $\tilde{K}$ is unique the claim that $D^{\tilde{w}}_{\tilde{X}}(\tilde{\Sigma}^n) \sim c_{\tilde{w}}\cdot(2g-2)^n$ for some nonzero $c_{\tilde{w}}$ actually holds for all $\tilde{w}$, regardless of the parity of $\tilde{w}\cdot\tilde{\Sigma}$ or its restriction to any embedded $S^1\times\tilde{\Sigma}$ (see \cite[Theorem 1]{Munoz-donaldson-nonsimple}).

Finally, since $K_{\tilde{X}} = K_X + \sum E_i$ where the $E_i$ are the exceptional divisors, we know from the blow-up formula \cite{FS-blowup} and the description of the basic classes of a blow-up \cite[Remark 3]{Munoz-donaldson-nonsimple} that $D^w_X$ is nonzero for any $w$ and that $K_X$ must be a basic class of $X$.
\end{proof}

\bibliographystyle{hplain}
\bibliography{references}

\end{document}